\documentclass[12pt,letterpaper,reqno]{amsart}

\usepackage{times}
\usepackage[T1]{fontenc}
\usepackage{mathrsfs}

\usepackage{latexsym}
\usepackage[dvips]{graphics}
\usepackage{epsfig}
\usepackage{amsmath,amsfonts,amsthm,amssymb,amscd}
\input amssym.def
\input amssym.tex

\makeindex

\newcommand\be{\begin{equation}}
\newcommand\ee{\end{equation}}
\newcommand\bea{\begin{eqnarray}}
\newcommand\eea{\end{eqnarray}}
\newcommand\bi{\begin{itemize}}
\newcommand\ei{\end{itemize}}
\newcommand\ben{\begin{enumerate}}
\newcommand\een{\end{enumerate}}
\newcommand\bc{\begin{center}}
\newcommand\ec{\end{center}}
\newcommand\ba{\begin{array}}
\newcommand\ea{\end{array}}

\newcommand{\Z}{\ensuremath{\mathbb{Z}}}
\newcommand{\Q}{\mathbb{Q}}

\newcommand{\js}[1]{ { \underline{#1} \choose p} }

\newtheorem{thm}{Theorem}[section]

\newtheorem{lem}[thm]{Lemma}

\newtheorem{rek}[thm]{Remark}

\newcommand{\ncr}[2]{{#1 \choose #2}}
\newcommand{\twocase}[5]{#1 \begin{cases} #2 & \text{#3}\\ #4
&\text{#5} \end{cases}   }

\newcommand{\gep}{\epsilon}

\newcommand{\must}{\mu_{\rm st}}

\numberwithin{equation}{section}

\begin{document}

\title{A combinatorial identity for studying Sato-Tate type problems}

\author{Steven J. Miller}\email{Steven.J.Miller@williams.edu}
\address{Department of Mathematics and Statistics, Williams College, Williamstown, MA 01267}

\author{M. Ram Murty}\email{murty@mast.queensu.ca}
\address{Department of Mathematics, Queen's University,
Kingston, Ontario, K7L 3N6, Canada}

\author{Frederick Strauch}\email{fws1@williams.edu}
\address{Department of Physics, Williams College, Williamstown, MA 01267}


\thanks{The first named author would like to thank Cameron and Kayla Miller for quietly sleeping on him while many of the calculations were done. Much of this paper was written when the first two authors attended the Graduate Workshop on $L$-functions and Random Matrix Theory at Utah Valley University in 2009, and it is a pleasure to thank the organizers. The first named author was partly supported by NSF grant DMS0600848. The second named author was partially supported by an NSERC Discovery grant.}

\maketitle

\begin{abstract} We derive a combinatorial identity which is useful in studying the distribution of Fourier coefficients of $L$-functions by allowing us to pass from knowledge of moments of the coefficients to the distribution of the coefficients.
\end{abstract}


\section{Introduction}

Recently M. Ram Murty and K. Sinha \cite{MS} proved effective equidistribution results showing the eigenvalues of Hecke operators on the space of cusp forms of weight $k$ and level $N$ agree with the Sato-Tate distribution. Their proof relied on bounding the discrepancy through an application of the Erd\"os-Turan inequality and estimates of exponential sums. In \cite{MM} the first two authors generalized their techniques to the Fourier coefficients of families of elliptic curves. The purpose of this note is to describe an interesting combinatorial identity needed in that analysis. 

We first describe the problem that motivated this work. Recall that if $E: y^2 = x^3 + Ax + B$ with $A, B \in \Z$ is an elliptic curve over $\Q$, the associated $L$-function is \be L(E,s)\ =\ \sum_{n=1}^\infty \frac{a_E(n)}{n^s} \ = \ \prod_p \left(1 - \frac{a_E(p)}{p^s} + \frac{\chi_0(p)}{p^{2s-1}}\right)^{-1},\ee with $\Delta=-16(4A^3+27B^2)$ the discriminant of $E$, $\chi_0$ the principal character modulo $\Delta$, and \bea a_E(p) & \ =\ & p - \#\{(x,y) \in (\Z/p\Z)^2:  y^2 \equiv x^3 + Ax + B \bmod p\} \nonumber\\ &=& - \sum_{x\bmod p} \js{x^3 + Ax + B}. \eea By Hasse's bound we know $|a_E(p)| \le 2\sqrt{p}$, so we may write $a_E(p) = 2\sqrt{p} \cos \theta_E(p)$, where we may choose $\theta_E(p) \in [0,\pi]$. The distribution of the $a_E(p)$'s are related to numerous problems of interest; for example, by the Birch and Swinnerton-Dyer conjecture the order of vanishing of $L(E,s)$ at the central point $s=1/2$ is conjecturally equal to the group of rational solutions of $E$. See \cite{Sil1,Sil2,ST} for more on elliptic curves.

In the analysis in \cite{MM}, one needs to understand sums of $\cos(m\theta_n)$, with $n$ ranging over a family of $L$-functions. Such estimates exist \cite{Ka,Mic,Ni}, and have been used by others to prove effective equidistribution results for two-parameter families of elliptic curves \cite{BS,Sh1,Sh2}. It is possible to avoid these estimates if instead one uses results of Birch \cite{Bi} for sums of the moments, i.e., sums of $\cos^r(\theta_n)$. While typically these lead to worse results, as there may be situations in future research where only the moments are known  we describe how one may prove effective equidistribution results concerning the distribution of the Fourier coefficients of $L$-functions using just the moments and combinatorics.

The key combinatorial ingredient in \cite{MM} is the following, which is the main result of this paper.

\begin{thm}\label{keycombinatoriallemmaforsuns} Let $m$ be an integer greater than or equal to 1. Then \be \twocase{\sum_{r=0}^m (-1)^r \ncr{m}{r} \ncr{m+r}{r} \frac1{(r+1)(m+r)} \ = \ }{1/2}{{\rm if} $m=1$}{0}{{\rm if} $m \ge 2$.} \ee \end{thm}

We prove this theorem in \S\ref{sec:combidentities}, and then discuss its application to effective equidistribution in \S\ref{sec:preliminaries}.


\section{Combinatorial Identities}\label{sec:combidentities}

Below we give two different proofs of Theorem \ref{keycombinatoriallemmaforsuns}, each highlighting a different approach to proving combinatorial identities. We first state some needed properties of the binomial coefficients. For $n, r$ non-negative integers we set $\ncr{n}{k} = \frac{n!}{k!(n-k)!}$. We generalize to real $n$ and $k$ a positive integer by setting \be \ncr{n}{k} \ = \ \frac{n(n-1)\cdots (n-(k-1))}{k!}, \ee which clearly agrees with our original definition for $n$ a positive integer and vanishes when $n$ is a non-negative integer less than $k$. Finally, we set $\ncr{n}{0} = 1$ and $\ncr{n}{k} = 0$ if $k$ is a negative integer.

To prove our main result we need the following two lemmas; we follow the proofs in \cite{Ward}.

\begin{lem}[Vandermonde's Convolution Lemma] Let $r,s$ be any two real numbers and $k, m, n$ integers. Then \be \sum_k \ncr{r}{m+k} \ncr{s}{n-k} \ = \ \ncr{r+s}{m+n}. \ee
\end{lem}

\begin{proof} Note that the summand is zero if either $m+k > r$ or $n-k > s$, and thus it is a finite sum over $k$. It suffices to prove the claim when $r, s$ are integers. The reason is that both sides are polynomials, and if the polynomials agree for an infinitude of integers then they must be identical. By changing $n$ and $k$, we see it suffices to consider the special case $m=0$, in which case we are reduced to showing \be\label{eq:ncrcombidentityvandermonde} \sum_k \ncr{r}{k} \ncr{s}{n-k} \ = \ \ncr{r+s}{n}. \ee Consider the polynomial \be\label{eq:usefulpolyvandermondeconv} (x+y)^r (x+y)^s\ =\ (x+y)^{r+s}.\ee If we use the binomial theorem to expand the left hand side of \eqref{eq:usefulpolyvandermondeconv}, we get the coefficient of the $x^n y^{r+s-n}$ is the left hand side of \eqref{eq:ncrcombidentityvandermonde}; this follows from looking at all the ways we could get an $x^n y^{r+s-n}$, which involves summing over the coefficients of $x^{k} y^{r-k}$ times the coefficients of $x^{n-k} y^{s-n+k}$. Similarly, if we use the binomial theorem we find the coefficient of $x^n y^{r+s-n}$ is the right hand side of \eqref{eq:usefulpolyvandermondeconv}. This proves \eqref{eq:ncrcombidentityvandermonde}, which completes the proof.
\end{proof}

\begin{lem}\label{lem:neededlemcomb} Let $\ell, m, s$ be non-negative integers. Then \be \sum_k (-1)^k \ncr{\ell}{m+k} \ncr{s+k}{n} \ = \ (-1)^{\ell+m} \ncr{s-m}{n-\ell}. \ee \end{lem}

\begin{proof} Using $\ncr{a}{b} = \ncr{a}{a-b}$, we rewrite $\ncr{s+k}{n}$ as $\ncr{s+k}{s+k-n}$, and we then rewrite $\ncr{s+k}{s+k-n}$ as $(-1)^{s+k-n} \ncr{-n-1}{s+k-n}$ by using the extension of the binomial coefficient, where we have pulled out all the negative signs in the numerators. The advantage of this simplification is that the summation index is now only in the denominator; further, the power of $-1$ is now independent of $k$. Factoring out the sign, our quantity is equivalent to \bea & & (-1)^{s-n} \sum_k \ncr{\ell}{m+k} \ncr{-n-1}{s+k-n} \ = \ (-1)^{s-n} \sum_k \ncr{\ell}{\ell-m-k} \ncr{-n-1}{s+k-n},\nonumber\\ \eea where we again use $\ncr{a}{b} = \ncr{a}{a-b}$. By Vandermonde's Convolution Lemma, this equals $(-1)^{s-n}$ $\ncr{\ell-n-1}{\ell-m-n+s}$. Using $\ncr{s-m}{\ell-m-n+s}  = \ncr{s-m}{n-\ell}$ and collecting powers of $-1$ completes the proof (note $(-1)^{\ell-m} = (-1)^{\ell+m}$).
\end{proof}

Using the above two lemmas, we can now prove our main result.

\begin{proof}[First Proof of Theorem \ref{keycombinatoriallemmaforsuns}] The case $m=1$ follows by direct evaluation. Consider now $m \ge 2$.
We have \bea S_m & \ = \ & \sum_{r=0}^m (-1)^r \ncr{m}{r} \ncr{m+r}{r} \frac1{(r+1)(m+r)} \nonumber\\ &= \ & \sum_{r=0}^m (-1)^r \ncr{m}{r} \frac{m+1}{m+1} \ncr{m+r}{r} \frac1{(r+1)(m+r)} \nonumber\\ &\ = \ &
\sum_{r=0}^m (-1)^r \frac{m! (m+1)}{(r+1)\cdot r!m!} \frac1{m+1} \ \frac{(m+r)(m+r-1)!}{r! m \cdot (m-1+r)!}  \frac1{m+r} \nonumber\\ &\ = \ & \sum_{r=0}^m (-1)^r \ncr{m+1}{r+1} \ncr{m-1+r}{r} \frac{1}{m(m+1)}  \nonumber\\ &\ = \ & \frac1{m(m+1)} \sum_{r=0}^m (-1)^r \ncr{m+1}{r+1} \ncr{m-1+r}{m-1}. \eea We change variables and set $u=r+1$; as $r$ runs from $0$ to $m$, $u$ runs from $1$ to $m+1$. To have a complete sum, we want $u$ to start at $0$; thus we add in the $u=0$ term, which is $\ncr{m-2}{m-1}$. As $m \ge 2$, this is 0 from the extension of the binomial coefficient (this is the first of two places where we use $m \ge 2$). Our sum $S_m$ thus equals \bea S_m & \ = \ & -\frac1{m(m+1)} \sum_{u=0}^{m+1} (-1)^u \ncr{m+1}{u} \ncr{m-2+u}{m-1}. \eea We now use Lemma \ref{lem:neededlemcomb} with $k=u$, $m=0$, $\ell = m+1$, $s=m-2$ and $n=m-1$; note the conditions of that lemma require $s$ to be a non-negative integer, which translates to our $m \ge 2$. We thus find \be S_m \ = \ -\frac1{m(m+1)} (-1)^{m+1} \ncr{m-2}{-2} \ = \ 0, \ee which completes the proof.
\end{proof}

We give another proof of Theorem \ref{keycombinatoriallemmaforsuns} below using hypergeometric functions, highlighting other approaches to proving combinatorial identities.

\begin{proof}[Second Proof of Theorem \ref{keycombinatoriallemmaforsuns}] Consider
the hypergeometric function \be\label{eq:hypergeodef} {}_2F_1(a,b,c;z) \ = \
\frac{\Gamma(c)}{\Gamma(b)\Gamma(c-b)} \int_0^1 \frac{t^{b-1} (1-t)^{c-b-
1}dt}{(1-tz)^a}. \ee The following identity for the normalization constant of the
Beta function is crucial in the expansions: \be B(x,y) \ = \ \int_0^1 t^{x-1} (1-
t)^{y-1}dt \ = \ \frac{\Gamma(x)\Gamma(y)}{\Gamma(x+y)}. \ee We can use the
geometric series formula to expand \eqref{eq:hypergeodef} as a power series in
$z$ involving Gamma factors,
\be \label{eq:hypergeodef2} {}_2F_1(a,b,c;z) \ = \ \frac{\Gamma(c)}{\Gamma(a)
\Gamma(b)} \sum_{n=0}^{\infty}
\frac{\Gamma(a+n)\Gamma(b+n)}{\Gamma(c+n)} \frac{z^n}{n!}. \ee
Rewriting $\ncr{m}{r}$ as $(-1)^r \ncr{r-m-1}{r}$,  $S_m$ can be written
\be S_m\ =\ \frac{1}{m! (-m-1)!} \sum_{r=0}^{\infty} \frac{(r-m-1)! (r+m-
1)!}{(r+1)!} \frac{1}{r!}, \ee where we have formally extended the series to
$\infty$ as the coefficients will vanish for $r \ge m+1$.  By comparing the two
infinite series and using the fact that $z! = \Gamma(z+1)$, we see that if we take $a=-m$,
$b=m$, $c=2$, $n=r$ and $z=1$, after some simple algebra we obtain
\be\label{eq:hypergeomprooflemmaa3} S_m\ =\ \frac{\Gamma(m) {}_2F_1(-
m,m,2;1)}{\Gamma(2) \Gamma(1+m)} \ = \ \frac{\Gamma(m)}{\Gamma(1+m)
\Gamma(2+m) \Gamma(2-m)}, \ee where the last step uses \be
{}_2F_1(a,b,c;1)\ =\ \frac{\Gamma(c) \Gamma(c-a-b)}{\Gamma(c-a)
\Gamma(c-b)}, \ee which follows from the normalization constant of the Beta
function. Note that the right hand side of \eqref{eq:hypergeomprooflemmaa3} is
$1/2$ when $m=1$ and $0$ for $m \ge 2$ because for such $m$,
$1/\Gamma(2-m) = 0$ due to the pole of $\Gamma(2-m)$.

\end{proof}

\begin{rek} It is also possible to prove Theorem \ref{keycombinatoriallemmaforsuns} through symbolic
manipulations. Using the results from \cite{PS,PSR}, one may input this into a
Mathematica package, which outputs a proof (though not all the steps).  The
reasoning behind this automated proof method is described in \cite{PWZ}, and many
of the identities for hypergeometric functions can be interpreted in a very
computational manner.  These results are also useful in random walk
calculations in physics (quantum and classical), and reduction to the
hypergeometric function is a convenient first step towards continuum limits or
long-time asymptotics.

\end{rek}


\section{Effective Equidistribution}\label{sec:preliminaries}

For a sequence of numbers $x_n$ modulo 1, a measure $\mu$ and an interval $I \subset [0,1]$, let \bea N_I(V_p) & \ = \ & \#\{n \le V_p: x_n \in I\} \nonumber\\ \mu(I) & \ =\ & \int_I \mu(t) dt. \eea The discrepancy $D_{I,V_p}(\mu)$ is \bea D_{I,V_p}(\mu) & \ =\  & \left|N_I(V_p) - V_p \mu(I)\right|; \eea with this normalization, the goal is to obtain the best possible estimate for how rapidly $D_{I,V_p}(\mu)/V_p$ tends to 0. A standard approach is to use exponential sums and the Erd\"os-Turan theorem. Modifying the ideas in \cite{MS} (see \cite{MM} for the details), one finds

\begin{thm}\label{thm:MSmainmodified} Let $\{x_n\} \subset [0,1]$ and let the notation be as above. Let $\{c_m\}$ be a sequence of numbers such that $\sum_{m=-\infty}^\infty |c_m|  <  \infty$. Let $||\mu|| = \sup_{x \in [0,1]} |F(x)|$ with $\mu = F(-x)dx$.  Then for any $V_p$ and $M$ the discrepancy satisfies \bea\label{eq:definitiondiscdistanceUS} D_{I,V_p}(\mu)  \le \frac{V_p||\mu||}{M+1}  + \sum_{1 \le m \le M} \left(\frac1{M+1} + \min\left(b-a,\frac{1}{\pi|m|}\right)\right) \left|\sum_{n=1}^{V_p} e(mx_n) - V_p c_m\right|. \nonumber\\ \eea  \end{thm}

Let $\must = F(-x)dx$ be the normalized Sato-Tate distribution on $[0,1]$. Its density is \be 2\sin^2(\pi x) \ = \ 1 - \frac12 \left(e(x) + e(-x)\right),\ee which implies that the coefficients of $\must$ are $c_0 = 1$, $c_{\pm 1} = -1/2$ and $c_m = 0$ for $|m| \ge 2$.

We consider the family of all elliptic curves modulo $p$ for $p \ge 5$. We may write these curves in Weierstrass form as $y^2 = x^3 - ax - b$ with $a,b \in \Z/p\Z$ and $4a^3 \neq 27b^2$. The number of pairs $(a,b)$ satisfying these conditions\footnote{If $a = 0$ then the only $b$ which is eliminated is $b=0$. If $a$ is a non-zero perfect square there are two $b$ that fail, while if $a$ is not a square than no $b$ fail. Thus the number of bad pairs of $(a,b)$ is $p$.} is \be V_p \ := \ p(p-1).\ee We use Birch's \cite{Bi} results on the moments of the family of all elliptic curves modulo $p$ (there are some typos in his explicit formulas; we correct these in \cite{MM}); unfortunately, these are results for quantities such as $(2 \sqrt{p} \cos \theta_n)^{2R}$, and the quantity which naturally arises when applying Theorem \ref{thm:MSmainmodified} is $e(mx_n)$ (with $x_n$ running over the normalized angles $\theta_{a,b}(p)/\pi$), specifically \be \left| \sum_{n = 1}^{V_p} e(mx_n) - V_p c_m \right|. \ee By applying some combinatorial identities we are able to rewrite our sum in terms of the moments, which allows us to use Birch's results. The point of this section is not to obtain the best possible error term but rather to highlight how one may generalize and apply the framework from \cite{MS}.

We first set some notation. Let $\sigma_k(T_p)$ denote the trace of the Hecke operator $T_p$ acting on the space of cusp forms of dimension $-2k$ on the full modular group. We have $\sigma_{k+1}(T_p) = O(p^{k + c + \gep})$, where from \cite{Sel} we see we may take $c = 3/4$ (there is no need to use the optimal $c$, as our final result, namely \eqref{eq:Mgivenc}, will yield the same order of magnitude result for $c=3/4$ or $c=0$). Let $\mathcal{M}_p(2R)$ denote the $2R$\textsuperscript{th} moment of $2\cos(\theta_n) = 2\cos(\pi x_n)$ (as we are concerned with the normalized values, we use slightly different notation than in \cite{Bi}): \be \mathcal{M}_p(2R) \ = \ \frac1{V_p} \sum_{n=1}^{V_p} \left(2 \cos(\pi  x_n)\right)^{2R}. \ee

\begin{lem}[Birch]\label{lem:birch} Notation as above, we have \be \mathcal{M}_p(2R) \ = \ \frac1{R+1}\ncr{2R}{R} + O\left(2^{2R} V_p^{-\frac{1-c-\gep}2}\right);\ee we may take $c = 3/4$ and thus there is a power saving.\footnote{Note $\frac1{R+1}\ncr{2R}{R}$ is the $R$\textsuperscript{th} Catalan number. The Catalan numbers are the moments of the semi-circle distribution, which is related to the Sato-Tate distribution by a simple change of variables.}
\end{lem}

\begin{proof} The result follows from dividing the equation for $S_R^\ast(p)$ on the bottom of page 59 of \cite{Bi} by $p^R$, as we are looking at the moments of the normalized Fourier coefficients of the elliptic curves, and then using the bound $\sigma_{k+1}(T_p) = O(p^{k+c+\gep})$, with $c = 3/4$ admissible by \cite{Sel}. Recall $V_p=p(p-1)$ is the cardinality of the family. We have \bea \mathcal{M}_p(2R) & \ = \ & \frac1{R+1}\ncr{2R}{R} \frac{p(p-1)}{V_p} \nonumber\\ & & \ \ + \ O\left(\sum_{k=1}^R \frac{2k+1}{R+k+1} \ncr{2R}{R+k} \frac{p^{1+c+\gep}}{V_p} + \frac{p}{p^R V_p}\right) \nonumber\\ & = & \frac1{R+1}\ncr{2R}{R} + O\left(2^{2R} V_p^{-\frac{1-c-\gep}2}\right)\eea since $V_p = p(p-1)$.
\end{proof}

A simple argument\footnote{To see that we may match the angles as claimed for the family of all elliptic curves, consider the elliptic curve $y^2 = x^3-ax-b$ with $4a^3\neq 27b^2$. Let $c$ be any non-residue modulo $p$, and consider the curve $y^2 = x^3 - ac^2x - bc^3$. Using the Legendre sum expressions for $a_E(p)$ and $a_{E'}(p)$, using the automorphism $x \to cx$ we see the second equals $\js{c}$ times the first; as we have chosen $c$ to be a non-residue, this means $2\sqrt{p}\cos(\theta_{E'}(p)) = -2\sqrt{p}\cos(\theta_E(p))$, or $\theta_{E'}(p)=\pi-\theta_E(p)$ as claimed.} shows that the normalized angles are symmetric about $1/2$. This implies \be \sum_{n=1}^{V_p} e(mx_n) \ = \ \sum_{n=1}^{V_p} \cos(2\pi mx_n) + i  \sum_{n=1}^{V_p} \sin(2\pi mx_n) \ = \ \sum_{n=1}^{V_p} \cos(2m\theta_n),\ee where the sine piece does not contribute as the angles are symmetric about $1/2$. Thus it suffices to show we have a power saving in \be \left| \sum_{n=1}^{V_p} \cos(2m\theta_n) - V_p c_m\right|. \ee By symmetry, it suffices to consider $m \ge 0$.

\begin{lem}Let $c_0 = 1$, $c_{\pm 1} = -1/2$ and $c_m = 0$ otherwise. There is some $c < 1$ such that \be \left| \sum_{n=1}^{V_p} \cos(2m\theta_n) - V_p c_m\right| \ \ll \ \left(m^2 2^{3m} V_p^{-\frac{1-c-\gep}2}\right); \ee by the work of Selberg \cite{Sel} we may take $c = 3/4$.
\end{lem}

\begin{proof}
The case $m=0$ is trivial. For $m=1$ we use the trigonometric identity $\cos(2\theta_n) = 2 \cos^2(\theta_n) - 1$. As $c_{\pm 1} = -1/2$ we have \bea \sum_{n=1}^{V_p} \cos(2\theta_n) - \frac{V_p}2 & \ = \ & \sum_{n=1}^{V_p} \left[ \left(2\cos^2 \theta_n - 1\right) + \frac12\right] \nonumber\\ & = & \frac12 \sum_{n=1}^{V_p} \left((2 \cos \theta_n)^2 - 1\right) \nonumber\\ & = & \frac12 \sum_{n=1}^{V_p} \left(\frac{(2\sqrt{p}\cos\theta_n)^2}{p} - 1\right). \eea Note the sum of $(2\sqrt{p}\cos \theta_n)^2$ is the second moment of the number of solutions modulo $p$. From \cite{Bi} we have that this is $p + O(1)$; the explicit formula given in \cite{Bi} for the second moment is wrong; see \cite{MM} for the correct statement. Substituting yields \bea \left|\sum_{n=1}^{V_p} \cos(2\theta_n) - \frac{V_p}2\right| & \ \ll \ & O(1). \eea

The proof is completed by showing that $\sum_{n=1}^{V_p} \cos(2m\theta_n) = O_m(V_p^{1/2})$ provided $2 \le m \le M$. In order to obtain the best possible results, it is important to understand the implied constants, as $M$ will have to grow with $V_p$ (which is of size $p^2$). While it is possible to analyze this sum for any $m$ by brute force, we must have $M$ growing with $p$, and thus we need an argument that works in general. As $c_{\pm 1} \neq 0$ but $c_m = 0$ for $|m| \ge 2$, we expect (and we will see) that the argument below does break down when $|m| = 1$.

There are many possible combinatorial identities we can use to express $\cos(2m\theta_n)$ in terms of powers of $\cos(\theta_n)$. We use the following (for a proof, see Definition 2 and equation (3.1) of \cite{Mil}): \be\label{eq:sjmexpansiontwocos2mtheta} 2\cos(2m\theta_n) \ = \ \sum_{r=0}^{m} c_{2m,2r} (2 \cos \theta_n)^{2r}, \ee where $c_{2r} = (2r)!/2$, $c_{0,0} = 0$, $c_{2m,0} = (-1)^m 2$ for $m\ge 1$, and for $1 \le r \le m$ set \be c_{2m,2r} \ = \ \frac{(-1)^{r+m}}{c_{2r}}\prod_{j=0}^{r-1} (m^2 - j^2) \ = \  \frac{(-1)^{m+r}}{c_{2r}}\frac{m \cdot (m+r-1)!}{(m-r)!}. \ee We now sum \eqref{eq:sjmexpansiontwocos2mtheta} over $n$ and divide by $V_p$, the cardinality of the family. In the argument below, at one point we replace $2^{2r}$ in an error term with $2012\frac1{r+1}\ncr{2r}{r} \cdot m^2$; this allows us to pull the $r$\textsuperscript{th} Catalan number, $\frac{1}{r+1}\ncr{2r}{r}$, out of the error term.\footnote{The reason this is valid is that the largest binomial coefficient is the middle (or the middle two when the upper argument is odd). Thus $2^{2r} = (1+1)^{2r} \le (2r+1) \ncr{2r}{r} \le 2(m+1)\ncr{2r}{r}$ (as $m \le r$), and the claim follows from $\frac{2012m^2}{r+1} \ge 2(m+1)$ for $m \ge 2$ and $0 \le r \le m$.} Using
Lemma \ref{lem:birch} we find \bea & & \frac1{V_p} \sum_{n=1}^{V_p} 2\cos(2m\theta_n)  \ = \  \sum_{r=0}^m c_{2m,2r} \frac1{V_p} \sum_{n=1}^{V_p} (2 \cos \theta_n)^{2r} \nonumber\\ & = & \sum_{r=0}^m \left( \frac1{r+1}\ncr{2r}{r} + O\left(2^{2r} V_p^{-\frac{1-c-\gep}2}\right) \right) c_{2m,2r}\nonumber\\ &=& \sum_{r=0}^m \left( \frac1{r+1}\frac{(2r)!}{r!r!} \frac{(-1)^{m+r} 2}{(2r)!}  \frac{m \cdot (m+r)!}{(m-r)! \cdot (m+r)} \right)\nonumber\\ & & \ \ \ \cdot \ \left(1 + O\left(m^2 V_p^{-\frac{1-c-\gep}2} \right)\right) \nonumber\\ & = & (-1)^m 2m \sum_{r=0}^m \left((-1)^r\frac{m!}{r!(m-r)!}   \frac{(m+r)!}{m!r!}\frac{1}{(r+1)(m+r)} \right)\nonumber\\ & & \ \ \ \cdot \ \left(1 + O\left(m^2 V_p^{-\frac{1-c-\gep}2} \right)\right)  \nonumber\\ & = & (-1)^m 2m \sum_{r=0}^m \left((-1)^r\ncr{m}{r}  \ncr{m+r}{r}\frac{1}{(r+1)(m+r)} \right)\nonumber\\ & & \ \ \ \cdot \ \left(1 + O\left(m^2 V_p^{-\frac{1-c-\gep}2} \right)\right). \eea We first bound the error term. For our range of $r$, $\ncr{m+r}{r} \le \ncr{2m}{m} \le 2^{2m}$. The sum of $\ncr{m}{r}$ over $r$ is $2^m$, and we get to divide by at least $m+r \ge m$. Thus the error term is bounded by \be O\left(m^2 2^{3m} V_p^{-\frac{1-c-\gep}2} \right). \ee We now turn to the main term. It it just $(-1)^m 2m$ times the sum in Theorem \ref{keycombinatoriallemmaforsuns}, which is shown in that theorem to equal 0 for any $|m| \ge 2$. Note that without Theorem \ref{keycombinatoriallemmaforsuns}, our combinatorial expansion would be useless.
\end{proof}

\begin{rek} It is possible to get a better estimate for the error term by a more detailed analysis of $\sum_{r\le m} \ncr{m}{r} \ncr{m+r}{r}$; however, the improved estimates only change the constants in the discrepancy estimates, and not the savings. This is because this sum is at least as large as the term when $r \approx m/2$, and this term contributes something of the order $3^{3m/2} / m$ by Stirling's formula. We will see that any error term of size $3^{am}$ for a fixed $a$ gives roughly the same value for the best cutoff choice for $M$, differing only by constants. Thus we do not bother giving a more detailed analysis to optimize the error here. \end{rek}

We now prove effective equidistribution for the family of all elliptic curves.

\begin{thm}\label{thm:allellipticcurves} For the family of all elliptic curves modulo $p$, as $p\to\infty$ we have \be D_{I,V_p}(\must) \ \le\ C \frac{V_p}{\log V_p} \ee for some computable $C$. \end{thm}

\begin{proof} We must determine the optimal $M$ to use in \eqref{eq:definitiondiscdistanceUS}: \bea  D_{I,V_p}(\must) & \ \ll \ & \frac{V_p}{M+1}  + \sum_{1 \le m \le M} \left(\frac1{M+1} + \frac1{m} \right) \left(m^2 2^{3m} V_p^{-\frac{1-c-\gep}2}\right) \nonumber\\ & \ll & \frac{V_p}{M} +  M 2^{3M} V_p^{-\frac{1-c-\gep}2} \nonumber\\ \eea as $\frac{1}{M+1} \ll \frac1{m}$ and $\sum_{m \le m} 2^{3m} \ll 2^{3M}$. For all $c > 0$ we find the minimum error by setting the two terms equal to each other, which yields \be V_p^{\frac{3-c-\gep}{2}} \ = \ M^2 2^{3M}. \ee For ease of exposition we replace $M^2 2^{3M}$ with $e^{3M}$; this worsens our constant slightly, but does not qualitatively change the result. Equating these errors means we are looking for $M$ such that \be e^{3M} \ = \ e^{\frac{3-c-\gep}{2} \log V_p}, \ee which implies \be\label{eq:Mgivenc} M \ = \ \frac{3-c-\gep}{6} \log V_p. \ee We thus see that we may find a constant $C$ such that \be D_{I,V_p}(\must) \ \le \ C \frac{V_p}{\log V_p}.\ee This yields a logarithm savings in the discrepancy, and proves effective equidistribution.
\end{proof}


\end{document}